\documentclass[12pt]{amsart}

\usepackage[english]{babel}
\usepackage{a4wide}

\usepackage{amssymb,amsfonts}
\usepackage{mathrsfs}
\usepackage{graphicx}
\usepackage[all,cmtip]{xy}
\usepackage{tikz}
\usepackage{tikz-cd}
\tikzset{dynkdot/.style={circle,draw,scale=.38}}
\usepackage{cases}
\usepackage{mathtools}
\usepackage{arydshln}
\usepackage{multirow}
\usepackage{amscd}
\usepackage{pb-diagram}

\usepackage{bm}
\usepackage{bbm}
\usepackage{dsfont}

\usepackage{hyperref}
\usepackage[alphabetic]{amsrefs}

\newtheorem{theorem}{Theorem}[section]

\newtheorem{proposition}[theorem]{Proposition}
\newtheorem{lemma}[theorem]{Lemma}

\newtheorem{corollary}[theorem]{Corollary}

\theoremstyle{definition}
\newtheorem{definition}[theorem]{Definition}
\newtheorem{example}[theorem]{Example}
\newtheorem{proposition-definition}[theorem]{Proposition-Definition}
\newtheorem{definition-theorem}[theorem]{Definition-Theorem}

\theoremstyle{remark}
\newtheorem{remark}[theorem]{Remark}

\numberwithin{equation}{section}

\newcommand{\opname}[1]{\operatorname{\mathsf{#1}}}

\newcommand{\Hom}{\opname{Hom}}

\tikzcdset{scale cd/.style={every label/.append style={scale=#1}, cells={nodes={scale=#1}}}}

\usepackage{appendix}

\begin{document}
\sloppy

\title[]{Linear independence of global monomials \\ on positive spaces}


\author{Peigen Cao}
\address{School of Mathematical Sciences, University of Science and Technology of China, Hefei, 230026, People's Republic of China
}
\email{peigencao@126.com}

\thanks{P. C. was supported by the National Key R\&D Program of China
(2024YFA1013801).}


\dedicatory{}


\date{}


\begin{abstract}
In this paper, we prove that global monomials on positive spaces are linearly independent, extending the basic fact that Laurent monomials in a Laurent polynomial algebra are linearly independent to a much more general setting.  We also establish a global monomial avoidance phenomenon for positive spaces. Our approach is based on the study of Newton polytopes of Laurent expansions. These general results apply to positive spaces arising from cluster algebras (including the totally sign-skew-symmetric case), $Y$-patterns, and Laurent phenomenon algebras whose clusters are related by subtraction-free birational transformations. In particular, we obtain the proper Laurent monomial property and the linear independence of cluster monomials for all cluster algebras and Laurent phenomenon algebras under consideration. Notably, the proper Laurent monomial property follows from the global monomial avoidance phenomenon for positive spaces.
\end{abstract}

\maketitle
\tableofcontents

\section{Introduction}
\subsection{Cluster algebras}
The notion of cluster algebras was introduced by Fomin and Zelevinsky \cite{fz_2002}  as a combinatorial approach to the dual canonical bases (or upper global bases) of quantum groups and to the theory of total positivity in algebraic groups. They often arise as the coordinate rings of various spaces of Lie-theoretic origin, e.g., double Bruhat cells \cite{bfz_2005}, Grassmannians \cite{Scott-2006}, unipotent subgroups \cite{gls_2011}, braid varieties \cite{CGGLSS-2025,GLSB-2026}.

A cluster algebra is a subalgebra of a rational function field generated by a special set of generators called {\em cluster variables}, which are grouped into overlapping subsets,
called {\em clusters}. Different clusters are related by sequences of {\em mutations}, which are {\em subtraction-free} birational transformations. One key remarkable property of cluster algebras is the {\em Laurent phenomenon}, which
says that any cluster variable is a Laurent polynomial when written as a rational function
in any other cluster.

A {\em cluster monomial} is a monomial in variables from the same cluster. When Fomin and Zelevinsky \cite{fz_2002} invented cluster algebras, one of their hopes was that the cluster monomials would be part of a `canonical' basis. This naturally led them to the conjecture
that cluster monomials are linearly independent (see \cite[Conjecture 4.16]{fz-2003-notes}). This conjecture has been verified for skew-symmetric cluster algebras in \cite{cklp-2013} via categorification, and for skew-symmetrizable cluster algebras in \cite{ghkk18} via scattering diagrams. Both proofs reduce to verifying the so-called {\em proper Laurent monomial property} \cite{cl-2012}, which states that if a cluster monomial $u$ is not a cluster monomial in a cluster ${\bf x}_{t_0}$, then its Laurent expansion with respect to ${\bf x}_{t_0}$ is a linear combination of proper Laurent monomials.

Laurent phenomenon (LP) algebras, introduced by Lam and Pylyavskyy \cite{lp-2016}, are analogous to cluster algebras but permit more flexible mutation relations. It has been proved that Laurent phenomenon algebras also enjoy the Laurent phenomenon, and it is conjectured that their cluster monomials are linearly independent, analogous to the corresponding conjecture of Fomin and Zelevinsky for cluster algebras. For 
graph LP algebras, a subclass of LP algebras of finite type defined from directed graphs, it is proved \cite{dtw-2025} that the cluster monomials form a basis, which is analogous to a result for cluster algebras of finite type by Caldero and Keller \cite{ck-2008}.

\subsection{Positive spaces and main results}
Since the mutation relations in cluster algebras are subtraction-free, it naturally fits into the framework of {\em positive spaces}, which we will explain below. The cluster monomials in cluster algebras play the role of global monomials on positive spaces. In this general framework, the linear independence of cluster monomials can be reformulated as the linear independence of global monomials on positive spaces.

To introduce positive spaces, we need the notion of semifield. Recall that a {\em semifield} $(\mathbb P, \cdot, \oplus)$ is  an abelian multiplicative group $(\mathbb P,  \cdot)$ endowed with an auxiliary addition $\oplus$ which is commutative, associative and satisfies that the multiplication  distributes over the auxiliary addition. For example, $$\mathbb R^{\rm max} \coloneqq   (\mathbb R,+, \max\{-,-\})$$ 
is a semifield, which is called a {\em tropical semifield}.

Let $\mathbb Q_{\rm sf}({\bf z})$ be the
set of all non-zero rational functions in ${\bf z}=(z_1, \ldots, z_r)$ that have subtraction-free expressions, namely, $\mathbb Q_{\rm sf}({\bf z})=\{\frac{F}{H}\mid 0\neq F,H\in\mathbb Z_{\geq 0}[z_1,\ldots,z_r]\}$. The set $\mathbb Q_{\rm sf}({\bf z})$  is a semifield
with respect to the usual operations of multiplication and addition. This semifield is called a {\em universal semifield}.

\begin{definition}[{\cite[Definition 2.3]{cdsgl-2023}, \cite[Section 1.1]{fg-2009}}]
    A {\em positive space} of dimension $r$ 
is a pair $\mathbb V = (\mathbb F, \,\{{\bf z}_t\}_{t \in \mathbb T})$, where
\begin{itemize}
    \item $\mathbb F$ is the field of rational functions in $r$ variables over $\mathbb Q$;
    \item $\mathbb T$ is any non-empty index set;
    \item $\{{\bf z}_t\}_{t \in \mathbb T}$ is an assignment of an ordered
set ${\bf z}_t$ of free generators of $\mathbb F$ over $\mathbb Q$ to each $t \in \mathbb T$, such that $\mathbb Q_{\rm sf}({\bf z}_t) = 
\mathbb Q_{\rm sf}({\bf z}_{t'}) \subseteq \mathbb F$ for all $t, t' \in \mathbb T$. 
\end{itemize}
We also write $\mathbb F = \mathbb F(\mathbb V)$ to indicate that $\mathbb F$ is the ambient field for the positive space $\mathbb V$.
\end{definition}
For a positive space $\mathbb V = (\mathbb F(\mathbb V), \{{\bf z}_t\}_{t \in \mathbb T})$, 
we call each ${\bf z}_t$ a {\em chart} of $\mathbb V$.  The elements $z_{1;t},\ldots,z_{r;t}$ in a chart ${\bf z}_t$ are called {\em coordinate variables}. 

\begin{definition}Let $\mathbb V = (\mathbb F(\mathbb V), \{{\bf z}_t\}_{t \in \mathbb T})$ be a positive space.
\begin{itemize}
    \item[(i)] The {\em coordinate algebra} $\mathcal U(\mathbb V)$ of $\mathbb V$ is defined to be the intersection of the following Laurent polynomial algebras:
    $$
    \mathcal U(\mathbb V)\coloneqq \bigcap_{t\in\mathbb T}\mathbb Z[z_{1;t}^{\pm 1},\ldots,z_{r;t}^{\pm 1}].$$ Elements in $\mathcal U(\mathbb V)$ are called the {\em global functions} on $\mathbb V$.
    \item[(ii)] The {\em ambient semifield} of $\mathbb V$ is defined as $\mathbb F_{>0}(\mathbb V) \coloneqq \mathbb Q_{\rm sf}({\bf z}_t)$  using any $t \in \mathbb T$.
    \item [(iii)] A {\em tropical point} on $\mathbb V$ is a semifield homomorphism $\beta:\mathbb F_{>0}(\mathbb V)\to \mathbb R^{\rm max}$. Denote by
    \[
    \mathbb V(\mathbb R^{\rm max})\coloneqq {\rm Hom}_{\rm sf}(\mathbb F_{>0}(\mathbb V), \, \mathbb R^{\rm max})
    \]
    the set of tropical points on $\mathbb V$.
\end{itemize}    
\end{definition}
We have the following facts.
\begin{itemize}
 \item Each chart ${\bf z}_t$ gives a bijection $\beta\mapsto (\beta(z_{1;t}),\ldots, \beta(z_{r;t}))^T$ from $ \mathbb V(\mathbb R^{\rm max})$ to $\mathbb R^r$.
\item A tropical point $\beta\in {\rm Hom}_{\rm sf}(\mathbb F_{>0}(\mathbb V), \, \mathbb R^{\rm max})$ can be naturally identified with its collection of coordinate vectors $\{ {\bf q}^t\in\mathbb R^r\mid t\in\mathbb T\}$, where ${\bf q}^t=(\beta(z_{1;t}),\ldots, \beta(z_{r;t}))^T$. 
   
\end{itemize}

\begin{definition}Let  $\mathbb V = (\mathbb F(\mathbb V), \{{\bf z}_t\}_{t \in \mathbb T})$ be a positive space. A Laurent monomial ${\bf z}_t^{\bf h}$ in a chart ${\bf z}_t$ is called a {\em global (Laurent) monomial} on $\mathbb V$ if ${\bf z}_t^{\bf h}\in \mathcal U(\mathbb V)$, i.e., if it is a global function on $\mathbb V$. Denote by $$\mathcal U^{\rm mono}(\mathbb V)\coloneqq \bigcup_{t\in\mathbb T_n}\{{\bf z}_t^{\bf h}\mid {\bf h}\in\mathbb Z^r\text{ and }{\bf z}_t^{\bf h}\in \mathcal U(\mathbb V)\}$$ the set of global monomials on $\mathbb V$, which is a subset of $\mathcal U(\mathbb V)\cap \mathbb F_{>0}(\mathbb V)$.
\end{definition}
We note that the set $\mathcal U^{\rm mono}(\mathbb V)$ is non-empty, since the constant $1$ is always a global monomial on $\mathbb V$. On the other hand, coordinate variables need not be global monomials in general (see Example~\ref{ex:Y-A2}).

Each non-zero global function $u\in \mathcal U(\mathbb V)$ gives rise to a family 
$$\{\mathsf P_u^t\subseteq \mathbb R^r\mid t\in\mathbb T\}$$ of polytopes in $\mathbb R^r$,
where $\mathsf P_u^t$ is the Newton polytope of the Laurent polynomial $u\in \mathbb Z[z_{1;t}^{\pm 1},\ldots,z_{r;t}^{\pm 1}]$  expressed in the chart ${\bf z}_t$. We keep this notation. The main results in this paper are as follows:

\begin{theorem}[Theorem \ref{thm:polytope}]
\label{intro-main-thm-1}
      Let $\mathbb V = (\mathbb F(\mathbb V), \{{\bf z}_t\}_{t \in \mathbb T})$ be a positive space.  Then the following conditions are equivalent for any two global monomials  $u,v\in \mathcal U^{\rm mono}(\mathbb V)$.
      \begin{itemize}
          \item [(a)] There exists a chart ${\bf z}_{t_0}$ of $\mathbb V$ such that $\mathsf P_u^{t_0}\subseteq \mathsf P_v^{t_0}$;
          \item[(b)] For any tropical point $\beta\in \mathbb V(\mathbb R^{\rm max})=\Hom_{\rm sf}(\mathbb F_{>0}(\mathbb V),\mathbb R^{\rm max})$, we have $\beta(u)\leq \beta(v)$;
          \item[(c)] For any chart  ${\bf z}_{t}$ of $\mathbb V$, we have   $\mathsf P_u^{t}\subseteq \mathsf P_v^{t}$;
          \item[(d)] $u=v$.
      \end{itemize}
\end{theorem}
A Laurent monomial ${\bf z}^{\bf h}=z_1^{h_1}\ldots z_r^{h_r}$ in $\mathbb Z[{\bf z}^{\pm 1}]=\mathbb Z[z_{1}^{\pm 1},\ldots,z_r^{\pm 1}]$ is said to be {\em proper}, if there exists some $i\in [1,r]\coloneqq\{1,2,\ldots,r\}$ such that $h_i<0$.
As a consequence of Theorem \ref{intro-main-thm-1}, we have the following corollary.
\begin{corollary}[Corollary \ref{cor:proper}]
\label{intro-main-cor}
      Let $\mathbb V = (\mathbb F(\mathbb V), \{{\bf z}_t\}_{t \in \mathbb T})$ be a positive space. Fix a chart ${\bf z}_{t_0}$ of $\mathbb V$, and let $M(t_0)\subseteq \mathcal U^{\rm mono}(\mathbb V)$ denote the set of global monomials in this chart. If $u\in \mathcal U^{\rm mono}(\mathbb V)$ is a global monomial on $\mathbb V$ and  $u\notin M(t_0)$, then the following hold.

\begin{itemize}
    \item[(i)] No global monomial in $M(t_0)$ appears in the Laurent expansion of $u$ with respect to the chart ${\bf z}_{t_0}$.

    \item[(ii)] If $\{{\bf z}_{t_0}^{\bf h}\mid {\bf h}\in \mathbb N^r\}\subseteq M(t_0)$\footnote{This condition does not hold for a general positive space, but it does hold for positive spaces arising
from cluster algebras and Laurent phenomenon algebras, thanks to the Laurent phenomenon for cluster
variables.}, then the Laurent expansion of $u$ with respect to ${\bf z}_{t_0}$ is a linear combination of proper Laurent monomials in ${\bf z}_{t_0}$.
\end{itemize}
\end{corollary}

We call the phenomenon in Corollary \ref{intro-main-cor} (i) the {\em global monomial avoidance phenomenon} for positive spaces.

\begin{theorem}[Theorem \ref{thm:linear-indp}]
\label{intro-main-thm-2}
  Let $\mathbb V = (\mathbb F(\mathbb V), \{{\bf z}_t\}_{t \in \mathbb T})$ be a positive space.  Then the global monomials on $\mathbb V$ are linearly independent over $\mathbb Z$. 
\end{theorem}

\begin{remark}
(i) The above results can be applied to the positive spaces arising from cluster algebras, $Y$-patterns, and Laurent phenomenon algebras whose clusters are related by subtraction-free birational transformations. In particular, our main results 
yield the proper Laurent monomial property of Laurent expansions and the linear independence of cluster monomials for all cluster algebras (see Corollary \ref{cor:A-side}) and for the Laurent phenomenon algebras  under consideration (see Corollary  \ref{cor:LP-side}).

(ii) Cluster algebras considered in this paper are  totally sign-skew-symmetric, which include skew-symmetrizable cluster algebras as a special case. Our results for cluster algebras hold in this general setting. The corresponding results for Laurent phenomenon algebras appear new to the best of our knowledge.

(iii) The idea and approach of this paper grew out of the study of the tropical invariant and the (partial) $F$-invariant in cluster algebras \cite{Cao-2023,cao-2026b}. These invariants are related to tropical evaluations and Newton polytopes of $F$-polynomials, which inspired the author to study Newton polytopes of Laurent expansions with respect to different charts/seeds.
\end{remark}

This paper is organized as follows. In Section \ref{sec:results}, after some preparations on polytopes and their support functions, we give the proofs of Theorem \ref{intro-main-thm-1}, Corollary \ref{intro-main-cor} and Theorem \ref{intro-main-thm-2}. In Section \ref{sec:application}, we give examples of positive spaces arising from cluster algebras, $Y$-patterns, and Laurent phenomenon algebras whose clusters are related by subtraction-free birational transformations. Our main results can be applied to these examples.

\section{Linear independence of global monomials via Newton polytopes}\label{sec:results}

\subsection{Polytopes and their support functions}
A {\em polytope} $\mathsf P$ in $\mathbb R^r$ is the convex hull of a finite (non-empty)
subset of $\mathbb R^r$, which is a bounded closed subset in $\mathbb R^r$. For a polytope  $\mathsf P$ in $\mathbb R^r$, its {\em support function} $h_{\mathsf P}:\mathbb R^r\to \mathbb R$ is defined by
\[ h_{\mathsf P}({\bf r}):=\max\{ \langle{{\bf a},{\bf r}}\rangle\mid {\bf a}\in\mathsf P\},\]
where $\langle-,-\rangle:\mathbb R^r\times \mathbb R^r\to \mathbb R$ is the standard inner product on $\mathbb R^r$.
It is known from \cite[Section 4.2]{BK-2012} or \cite[Section 1.7]{Schneider_2013} that a polytope $\mathsf P$ can be recovered from its support function $h_{\mathsf P}$ by 

\begin{eqnarray}\label{eqn:P-h}
     \mathsf P=\{{\bf a}\in\mathbb R^r\mid \langle{{\bf a},{\bf r}}\rangle\leq h_{\mathsf P}({\bf r}),\;\forall {\bf r}\in\mathbb R^r\}.
\end{eqnarray}

Let $\mathsf P_1$ and $\mathsf P_2$ be two polytopes in $\mathbb R^r$. The {\em Minkowski sum} of $\mathsf P_1$ and $\mathsf P_2$ is the polytope in $\mathbb R^r$ given by
\[\mathsf P_1+\mathsf P_2:=\{ {\bf a}+{\bf b}\mid {\bf a}\in \mathsf P_1,\;{\bf b}\in\mathsf P_2\}.\]

\begin{theorem}[{\cite[Theorem 1.7.5]{Schneider_2013}}] \label{thm:h-additive}
Let $\mathsf P_1$ and $\mathsf P_2$ be two polytopes in $\mathbb R^r$. Then \[h_{\mathsf P_1+\mathsf P_2}=h_{\mathsf P_1}+h_{\mathsf P_2}.\]
\end{theorem}

Given a non-zero Laurent polynomial $L({\bf z})=\sum_{{\bf h}\in\mathbb Z^r}c_{\bf h}{\bf z}^{\bf h}\in\mathbb Z[z_1^{\pm 1},\ldots,z_r^{\pm 1}]$, its {\em Newton polytope} $\mathsf P(L)$ is defined as the convex hull of the finite set $\{{\bf h}\in\mathbb Z^r\mid c_{\bf h}\neq 0\}$ in $\mathbb R^r$.

\begin{proposition}[{\cite[Chapter 6, Prop. 1.2]{GKZ-1994}}]
\label{pro:GKZ}
 Let $L_1$ and $L_2$ be two non-zero Laurent polynomials in $\mathbb Z[z_1^{\pm1},\ldots,z_r^{\pm 1}]$. Then 
 $$\mathsf P(L_1L_2)=\mathsf P(L_1)+\mathsf P(L_2),$$
 where $\mathsf P(L_k)$ is the Newton polytope of $L_k$ for $k=1,2$.
\end{proposition}

\subsection{Proofs of the main results}
In this subsection, we give the proofs of the main results after some preparations.

\begin{lemma}\label{lem:b-h}
     Let $L\in\mathbb Z[z_1^{\pm 1},\ldots,z_r^{\pm 1}] \cap \mathbb Q_{\rm sf}({\bf z})$  and let $h_{\mathsf P(L)}$ be the support function of 
      the Newton polytope $\mathsf P(L)$ of $L$.  Then for any semifield homomorphism $\beta\in \Hom_{\rm sf}(\mathbb Q_{\rm sf}({\bf z}),\mathbb R^{\rm max})$, we have 
\[\beta(L)=h_{\mathsf P(L)}({\bf q}),\]
where ${\bf q}:=(\beta(z_1),\ldots, \beta(z_r))^T\in\mathbb R^r$.
\end{lemma}

\begin{proof}
We first consider the case that $L=\sum_{{\bf h}\in\mathbb Z^r}c_{\bf h}{\bf z}^{\bf h}\in \mathbb Z_{\geq 0}[z_1^{\pm 1},\ldots,z_r^{\pm 1}]$ is a Laurent polynomial with non-negative coefficients. Since $\beta\in \Hom_{\rm sf}(\mathbb Q_{\rm sf}({\bf z}),\mathbb R^{\rm max})$ and ${\bf q}=\beta({\bf z})^T=(\beta(z_1),\ldots,\beta(z_r))^T$, we have
\begin{eqnarray*}
    \beta(L)&=&\max\{\beta({\bf z}^{\bf h})\mid c_{\bf h}\neq 0\}=
\max\{\langle {\bf q}, {\bf h}\rangle\mid c_{\bf h}\neq 0\}\\
&=&\max\{\langle {\bf q}, {\bf h}\rangle\mid {\bf h}\in \mathsf P(L)\}=h_{\mathsf P(L)}({\bf q}).
\end{eqnarray*}

Now we turn to the general case $L\in \mathbb Z[z_1^{\pm 1},\ldots,z_r^{\pm 1}] \cap \mathbb Q_{\rm sf}({\bf z})$. In this case, there exist $0\neq F,H\in\mathbb Z_{\geq 0}[z_1,\ldots,z_r]$ such that $L=F/H\in \mathbb Q_{\rm sf}({\bf z})$.
By applying the semifield homomorphism $\beta$, we have $$\beta(L)=\beta(F)-\beta(H).$$ Since $0\neq F,H\in\mathbb Z_{\geq 0}[z_1,\ldots,z_r]$ and by what we have proved, we have  
\[\beta(F)=h_{\mathsf P(F)}({\bf q})\quad \text{and}\quad \beta(H)=h_{\mathsf P(H)}({\bf q}).
\]

By viewing $F=L\cdot H$ as an equality in $\mathbb Z[z_1^{\pm 1},\ldots, z_r^{\pm 1}]$ and by Proposition \ref{pro:GKZ}, we have $\mathsf P(F)=\mathsf P(L)+\mathsf P(H)$. Then by Theorem \ref{thm:h-additive}, we have
 \[ h_{\mathsf P(F)}({\bf q})=h_{\mathsf P(L)}({\bf q})+h_{\mathsf P(H)}({\bf q}).
 \]
 Thus we have that $h_{\mathsf P(L)}({\bf q})=h_{\mathsf P(F)}({\bf q})-h_{\mathsf P(H)}({\bf q})=\beta(F)-\beta(H)=\beta(L)$.
\end{proof}

\begin{lemma}\label{lem:p-h-sub}
Let $L_1, L_2\in \mathbb Z[z_1^{\pm1},\ldots,z_r^{\pm 1}]\cap \mathbb Q_{\rm sf}({\bf z})$ and let $\mathsf P(L_1), \mathsf P(L_2)$ be their Newton polytopes. Then the following conditions are equivalent.
\begin{itemize}
    \item[(a)] $\mathsf P(L_1)\subseteq \mathsf P(L_2)$;
    \item[(b)] $h_{\mathsf P(L_1)}({\bf q})\leq h_{\mathsf P(L_2)}({\bf q})$ for any ${\bf q}\in\mathbb R^r$;
    \item[(c)] $\beta(L_1)\leq\beta(L_2)$ for any $\beta\in \Hom_{\rm sf}(\mathbb Q_{\rm sf}({\bf z}),\mathbb R^{\rm max})$.
\end{itemize}
\end{lemma}
\begin{proof} The equivalence of (a) and (b) follows from the definition of support functions and \eqref{eqn:P-h}. 
The equivalence of (b) and (c) follows from Lemma \ref{lem:b-h}.
\end{proof}

 Let $\mathbb V = (\mathbb F(\mathbb V), \{{\bf z}_t\}_{t \in \mathbb T})$ be a positive space. Recall that  
each non-zero global function $u\in \mathcal U(\mathbb V)$ gives rise to a family 
$$\{\mathsf P_u^t\subseteq \mathbb R^r\mid t\in\mathbb T\}$$ of polytopes in $\mathbb R^r$,
where $\mathsf P_u^t$ is the Newton polytope of the Laurent polynomial $u\in \mathbb Z[z_{1;t}^{\pm 1},\ldots,z_{r;t}^{\pm 1}]$ expressed in the chart ${\bf z}_t$.

\begin{theorem}\label{thm:polytope}
      Let $\mathbb V = (\mathbb F(\mathbb V), \{{\bf z}_t\}_{t \in \mathbb T})$ be a positive space. Then the following conditions are equivalent for any two global monomials  $u,v\in \mathcal U^{\rm mono}(\mathbb V)$.
      \begin{itemize}
          \item [(a)] There exists a chart ${\bf z}_{t_0}$ of $\mathbb V$ such that $\mathsf P_u^{t_0}\subseteq \mathsf P_v^{t_0}$;
          \item[(b)] For any tropical point $\beta\in \mathbb V(\mathbb R^{\rm max})=\Hom_{\rm sf}(\mathbb F_{>0}(\mathbb V),\mathbb R^{\rm max})$, we have $\beta(u)\leq \beta(v)$;
          \item[(c)] For any chart  ${\bf z}_{t}$ of $\mathbb V$, we have   $\mathsf P_u^{t}\subseteq \mathsf P_v^{t}$;
          \item[(d)] $u=v$.
      \end{itemize}
\end{theorem}

\begin{proof}

(a)$\Rightarrow$(b): For a chart ${\bf z}_t$, we denote by $u=L_u^t({\bf z}_t)\in\mathbb Z[z_{1;t}^{\pm 1},\ldots,z_{r;t}^{\pm 1}]$ the Laurent expansion of $u$ with respect to the chart ${\bf z}_t$.
 Suppose $\mathsf P_u^{t_0}\subseteq \mathsf P_v^{t_0}$ for some chart ${\bf z}_{t_0}$.  Then by Lemma \ref{lem:p-h-sub}, for any $\beta\in \mathbb V(\mathbb R^{\rm max})=\Hom_{\rm sf}(\mathbb Q_{\rm sf} ({\bf z}_{t_0}),\mathbb R^{\rm max})$,
we have
\[\beta(u)=\beta(L_u^{t_0})\leq \beta(L_v^{t_0})=\beta(v).
\]
This proves  the implication (a)$\Rightarrow$(b). A similar argument, using the equivalence in Lemma~\ref{lem:p-h-sub}, gives  the implication (b)$\Rightarrow$(c). The implication (d)$\Rightarrow$(a) is clear.

It remains to prove (c)$\Rightarrow$(d). Since $v$ is a global monomial on $\mathbb V$, it is of the form $v={\bf z}_t^{\bf h}$ for some chart ${\bf z}_t$ and ${\bf h}\in \mathbb Z^r$. By applying the condition (c) to this chart, we have $\emptyset\neq \mathsf P_u^t\subseteq \mathsf P_v^t=\{{\bf h}\}$. We obtain that $\mathsf P_u^t=\{{\bf h}\}$, and hence $$u=\lambda {\bf z}_t^{\bf h}=\lambda v$$ for some $0\neq \lambda\in\mathbb Z$. Since $u\in\mathbb Q_{\rm sf}({\bf z}_t)$, we have $\lambda\in\mathbb Z_{\geq 1}$. Since $u$ is also a global monomial on $\mathbb V$, there exist a chart ${\bf z}_{t'}$ and ${\bf h}'\in\mathbb Z^r$ such that $u={\bf z}_{t'}^{{\bf h}'}$. 
Thus we have $\lambda v=u={\bf z}_{t'}^{{\bf h}'}$. Then $v=\lambda^{-1}{\bf z}_{t'}^{{\bf h}'}$ can be viewed as  the Laurent expansion of the global monomial $v$ with respect to the chart ${\bf z}_{t'}$, which implies that $\lambda^{-1}\in\mathbb Z$. Since $\lambda\in\mathbb Z_{\geq 1}$,
we must have $\lambda=1$. Thus $u={\bf z}_t^{\bf h}=v$. This proves the implication (c)$\Rightarrow$(d).
\end{proof}

\begin{remark}
We note that the equivalence of the conditions (a), (b), and (c) in fact holds for any two elements $u,v$ in $\mathcal U(\mathbb V)\cap \mathbb F_{>0}(\mathbb V)$. 
\end{remark}

Recall that a Laurent monomial ${\bf z}^{\bf h}=z_1^{h_1}\ldots z_r^{h_r}$ in $\mathbb Z[{\bf z}^{\pm 1}]=\mathbb Z[z_{1}^{\pm 1},\ldots,z_r^{\pm 1}]$ is said to be {\em proper}, if there exists some $i$ such that $h_i<0$.

\begin{corollary}\label{cor:proper}
      Let $\mathbb V = (\mathbb F(\mathbb V), \{{\bf z}_t\}_{t \in \mathbb T})$ be a positive space. Fix a chart ${\bf z}_{t_0}$ of $\mathbb V$, and let $M(t_0)\subseteq \mathcal U^{\rm mono}(\mathbb V)$ denote the set of global monomials in this chart. If $u\in \mathcal U^{\rm mono}(\mathbb V)$ is a global monomial on $\mathbb V$ and  $u\notin M(t_0)$, then the following hold.

\begin{itemize}
    \item[(i)] No global monomial in $M(t_0)$ appears in the Laurent expansion of $u$ with respect to the chart ${\bf z}_{t_0}$.

    \item[(ii)] If $\{{\bf z}_{t_0}^{\bf h}\mid {\bf h}\in \mathbb N^r\}\subseteq M(t_0)$, then the Laurent expansion of $u$ with respect to ${\bf z}_{t_0}$ is a linear combination of proper Laurent monomials in ${\bf z}_{t_0}$.
\end{itemize}
\end{corollary}
\begin{proof}
(i) For any global monomial $v={\bf z}_{t_0}^{{\bf h}_0}\in M(t_0)$, it suffices to show that ${\bf h}_0\notin \mathsf P_u^{t_0}$. Assume by contradiction that ${\bf h}_0\in \mathsf P_u^{t_0}$. Then we have $$\mathsf P_v^{t_0}=\{{\bf h}_0\}\subseteq \mathsf P_u^{t_0}$$ for the global monomials $u$ and $v$. Then by Theorem \ref{thm:polytope}, we obtain $u=v={\bf z}_{t_0}^{{\bf h}_0}\in M(t_0)$. This contradicts $u\notin M(t_0)$. Hence, ${\bf h}_0\notin \mathsf P_u^{t_0}$. In particular, ${\bf z}_{t_0}^{{\bf h}_0}$ does not appear in Laurent expansion of $u$ with respect to ${\bf z}_{t_0}$. Since ${\bf z}_{t_0}^{{\bf h}_0}$ is an arbitrary global monomial in chart ${\bf z}_{t_0}$, it follows that the Laurent expansion of $u$ with respect to ${\bf z}_{t_0}$ contains no global monomial from ${\bf z}_{t_0}$.

(ii) This follows from (i), since  $\{{\bf z}_{t_0}^{\bf h}\mid {\bf h}\in \mathbb N^r\}\subseteq M(t_0)$. 
\end{proof}

For future reference, we call the phenomenon in Corollary \ref{cor:proper} (i) the {\em global monomial avoidance phenomenon} for positive spaces.

\begin{theorem}\label{thm:linear-indp}
  Let $\mathbb V = (\mathbb F(\mathbb V), \{{\bf z}_t\}_{t \in \mathbb T})$ be a positive space.  Then the global monomials on $\mathbb V$ are linearly independent over $\mathbb Z$. 
\end{theorem}

\begin{proof}
Suppose that we have a linear relation $\sum_{i=1}^m c_i u_i=0$,
where $c_i\in\mathbb Z$ and $u_1,\ldots,u_m$ are distinct global monomials. For the global monomial $u_1$, and there exist a chart ${\bf z}_{t_1}$ and ${\bf h}_1\in\mathbb Z^r$ such that $u_1={\bf z}_{t_1}^{{\bf h}_1}$. Thus 
\[
\mathsf P_{u_1}^{t_1}=\{{\bf h}_1\}.
\]
For any other global monomial $u_j$ with $j\in[2,m]$, the vector ${\bf h}_1$ can not belong to $\mathsf P_{u_j}^{t_1}$. Otherwise
\[
\mathsf P_{u_1}^{t_1}=\{{\bf h}_1\}\subseteq \mathsf P_{u_j}^{t_1}.
\]
Then Theorem \ref{thm:polytope} implies that $u_1=u_j$. This contradicts that $u_1$ and $u_j$ are different. In particular, the Laurent expansion of $u_j$ with respect to the chart ${\bf z}_{t_1}$  has zero coefficient at the Laurent monomial ${\bf z}_{t_1}^{{\bf h}_1}$ for any $j\in [2,m]$.

Therefore, in the Laurent expansion of $\sum_{{\bf h}\in\mathbb Z^r}\lambda_{\bf h}{\bf z}_{t_1}^{\bf h}$ of $\sum_{i=1}^m c_i u_i$ with respect to the chart ${\bf z}_{t_1}$, the coefficient $\lambda_{{\bf h}_1}$ of ${\bf z}_{t_1}^{{\bf h}_1}$ is exactly $c_1$. Since $\sum_{{\bf h}\in\mathbb Z^r}\lambda_{\bf h}{\bf z}_{t_1}^{\bf h}=\sum_{i=1}^m c_i u_i=0$, we have $\lambda_{\bf h}=0$ for any ${\bf h}\in\mathbb Z^r$. In particular, $c_1=\lambda_{{\bf h}_1}=0$.

By applying the same argument, we obtain $$c_1=\cdots=c_m=0.$$ Hence, the global monomials on $\mathbb V$ are linearly independent over $\mathbb Z$. 
\end{proof}

We note that the linear independence of global monomials on positive spaces can be viewed as a generalization of the basic fact that the Laurent monomials in $\mathbb Z[z_1^{\pm 1},\ldots, z_r^{\pm 1}]$ are linearly independent. To see this, let us consider the positive space $\mathbb V$ given by  $\mathbb F(\mathbb V)=\mathbb Q(z_1,\ldots,z_r)$ with a single chart ${\bf z}=(z_1,\ldots,z_r)$. Then $\mathcal U(\mathbb V)=\mathbb Z[z_1^{\pm 1},\ldots, z_{r}^{\pm 1}]$ and the global monomials on $\mathbb V$ are exactly the Laurent monomials in $z_1,\ldots,z_r$. In the next section, we will give more non-trivial examples of positive spaces.

\section{Examples and applications}\label{sec:application}
In this section, we give some natural examples of positive spaces from cluster theory. Our general results on positive spaces can be applied to these examples.
\subsection{Mutation of matrices}
The notion of mutation of matrices is introduced by Fomin and Zelevinsky \cite{fz_2002}, which plays an important role in cluster theory.
\begin{definition}
    Let $A=(a_{ij})$ be an $m\times n$ integer matrix. For any integer $k$ with $1\leq k\leq {\rm min}\{m,n\}$, the  {\em mutation} of $A$
   in direction $k$ is
defined to be the new integer matrix $\mu_k(A)=A'=(a_{ij}')$ given by
\begin{eqnarray}\label{eqn:b-mutation}
a_{ij}^\prime&=&\begin{cases}-a_{ij}, & \text{if}\;i=k\;\text{or}\;j=k,\\
 a_{ij}+[a_{ik}]_+[a_{kj}]_+-[-a_{ik}]_+[-a_{kj}]_+,&\text{otherwise},\end{cases}\nonumber
\end{eqnarray}
where $[a]_+:=\max\{a,0\}$ for any $a\in\mathbb R$.
\end{definition}

It is easy to check that $\mu_k^2(A)=A$. An $n\times n$ integer matrix $B=(b_{ij})$ is said to be 
 \begin{itemize}
     \item {\em sign-skew-symmetric}, if for any $i$ and $j$, either $b_{ij}=b_{ji}=0$ or $b_{ij}b_{ji}<0$;
     \item {\em totally sign-skew-symmetric}, if $B$ is sign-skew-symmetric and $\overset{\leftarrow}{\mu}(B)$ remains sign-skew-symmetric for  any sequence $\overset{\leftarrow}{\mu}$ of mutations;
     \item {\em skew-symmetrizable}, if there exists a diagonal integer matrix $D=diag(d_1,\ldots,d_n)$ with each $d_i>0$ such that $DB$ is skew-symmetric. Such a diagonal matrix $D$ is called a {\em skew-symmetrizer} of $B$.
 \end{itemize}
 
Clearly, every skew-symmetrizable matrix is sign-skew-symmetric. The following result shows that, in fact, skew-symmetrizable matrices are totally sign-skew-symmetric.

\begin{proposition}[{\cite{fz_2002}}]
  If $B=(b_{ij})_{n\times n}$ is skew-symmetrizable, then $B':=\mu_k(B)$ is still skew-symmetrizable and the two matrices $B, B'$ share the same skew-symmetrizers. 
\end{proposition}

\subsection{Positive spaces from cluster algebras}
In this subsection, we fix a pair $(n,m)$ of integers with $0<n\leq  m$. Let $\mathbb F$ be the field of rational functions in $m$ variables over $\mathbb Q$.

 A {\em  seed} of rank $n$ in $\mathbb F$ is a pair
$({\bf x}, \widetilde B)$, where
\begin{itemize}
	\item ${\bf x} = (x_1, \ldots, x_{m})$ is an ordered set of free generators of $\mathbb F$ over $\mathbb Q$;
	\item  $\widetilde B=\begin{bmatrix}
       B\\ P
   \end{bmatrix}=(b_{ij})$ is an $m\times n$ matrix such that its top $n\times n$ submatrix $B$ is totally sign-skew-symmetric matrix\footnote{Here we follow the setting of \cite{fz_2002}, where cluster algebras are defined for totally sign-skew-symmetric matrices.}.
\end{itemize}
In this case, the ordered set ${\bf x}$ is called the {\it cluster} of $({\bf x}, \widetilde B)$. The elements of ${\bf x}$ are called
{\it cluster variables}. More precisely, we call $x_1,\ldots,x_n$ {\em unfrozen cluster variables} and $x_{n+1},\ldots,x_{m}$ {\em frozen (cluster) variables}.
The matrices $B$ and $P$ are respectively called the {\it  exchange matrix} and the {\it coefficient matrix} of $({\bf x}, \widetilde B)$.

 The {\em ${\bf A}$-mutation} of a seed $({\bf x}, \widetilde B)$ of rank $n$ in direction $k\in[1,n]$ is the pair  $({\bf x}',  \widetilde  B')=\mu_k({\bf x},  \widetilde B)$ given by $\widetilde B'=\mu_k(\widetilde B)$ and
\begin{eqnarray}
\label{eqn:x-mutation}
 x_i^\prime=\begin{cases}x_i,&
 \text{if}\;i\neq k;\\
 x_k^{-1}\cdot (\prod_{j=1}^mx_j^{[b_{jk}]_+}+\prod_{j=1}^mx_j^{[-b_{jk}]_+}),&\text{if}\;i= k.\end{cases}
\end{eqnarray}
One can check that the new pair  $({\bf x}', \widetilde B')=\mu_k({\bf x}, \widetilde B)$ is still a seed in $\mathbb F$ and we have $\mu_k^2({\bf x},  \widetilde B)=({\bf x}, \widetilde B)$. Since the relation in \eqref{eqn:x-mutation} is subtraction-free , we have $\mathbb Q_{\rm sf}({\bf x})=\mathbb Q_{\rm sf}({\bf x}')$.

Let $\mathbb T_n$ denote the $n$-regular tree. We
 label the edges of $\mathbb T_n$ by $1,\ldots, n$ such that the $n$ different edges adjacent to the same vertex of $\mathbb T_n$ receive different labels.

\begin{definition}[Cluster pattern] A {\em cluster pattern} $\mathcal S_X=\{({\bf x}_t,  \widetilde B_t)\mid t\in \mathbb T_n\}$ in $\mathbb F$
	is an assignment of a seed $({\bf x}_t, \widetilde B_t)$ to
 	every vertex $t$ of $\mathbb T_n$ such that $({\bf x}_{t'},  \widetilde B_{t'})$ is the ${\bf A}$-mutation of $({\bf x}_t,  \widetilde B_t)$ in direction $k$ whenever
	\begin{xy}(0,1)*+{t}="A",(10,1)*+{t'}="B",\ar@{-}^k"A";"B" \end{xy} in $\mathbb T_n$.
\end{definition}
We note that the frozen variables are the same in any two clusters of $\mathcal S_X$.

\begin{definition}
    Let  $\mathcal S_X=\{({\bf x}_t,  \widetilde B_t)\mid t\in \mathbb T_n\}$ be a cluster pattern  in $\mathbb F$.
    \begin{itemize}
    \item [(i)] The {\em cluster algebra} $\mathcal A$ associated with the cluster pattern $\mathcal S_X$ is the $\mathbb Z$-subalgebra of $\mathbb F$ generated by all the cluster variables (frozen and unfrozen), together with the inverse of frozen variables, i.e.,
 $$\mathcal A\coloneqq \mathbb Z[x_{1;t},\ldots,x_{n;t},x_{n+1;t}^{\pm 1},\ldots,x_{m;t}^{\pm 1}\mid t\in\mathbb T_n].$$
 \item[(ii)] The {\em upper cluster algebra} $\mathcal U$ associated with the cluster pattern $\mathcal S_X$ is defined by the following intersection of Laurent polynomial algebras: 
 \[
 \mathcal U\coloneqq \bigcap_{t\in\mathbb T_n}\mathbb Z[x_{1;t}^{\pm 1},\ldots,x_{m;t}^{\pm 1}].\]
 \item [(iii)] The {\em positive ${\bf A}$-space} associated with the cluster pattern $\mathcal S_X$ is the positive space given by
 \[\mathbb A\coloneqq (\mathbb F, \{{\bf x}_t\}_{t\in\mathbb T_n}).\]
\end{itemize}
\end{definition}

We see that the coordinate algebra $\mathcal U(\mathbb A)$ of a positive ${\bf A}$-space $\mathbb A$ is exactly the corresponding upper cluster algebra.

\begin{theorem}[{\cite{fz_2002}, Laurent phenomenon}] Let $\mathcal A$ be a cluster algebra and $z$ a cluster variable of $\mathcal A$. Then for any seed $({\bf x}_t, \widetilde B_t)$ of $\mathcal A$, we have $z\in\mathbb Z[x_{1;t}^{\pm 1},\ldots, x_{m;t}^{\pm 1}]$. In particular, $\mathcal A$ is contained in the corresponding upper cluster algebra $\mathcal U$.
\end{theorem}

\begin{example}
\label{ex:A-A2}
   Take $B=\begin{bmatrix}
    0&1\\-1&0
\end{bmatrix}$ and ${\bf x}=(x_1,x_2)$. 
It is easy to  check that the  cluster pattern defined by the initial seed $({\bf x}, B)$  has only five  clusters up to permutation: \[(x_1,x_2),\;(x_2,x_3),\;(x_3,x_4),\;(x_4,x_5),\;(x_5,x_1),\quad \text{where}\]
\[x_3:=   \frac{x_2+1}{x_1}, \;x_4:=   \frac{x_1+x_2+1}{x_1x_2}, \;x_5:=  \frac{x_1+1}{x_2}.\]
\end{example}

A monomial in cluster variables from the same cluster is called a {\em cluster monomial}.  Our general results on positive spaces yield the following results for cluster algebras.

\begin{corollary}\label{cor:A-side}
Let $\mathcal A$ be a (totally sign-skew-symmetric) cluster algebra, and let $\mathbb A=(\mathbb F, \{{\bf x}_t\}_{t\in\mathbb T_n})$  be the corresponding positive ${\bf A}$-space. Fix a cluster ${\bf x}_{t_0}$. Then:
\begin{itemize}
    \item[(i)] If $u$ is a cluster monomial and  $u\notin \{{\bf x}_{t_0}^{\bf h}\mid {\bf h}\in\mathbb N^m\}$, then the Laurent expansion of $u$ with respect to ${\bf x}_{t_0}$ is a linear combination of proper Laurent monomials in ${\bf x}_{t_0}$.
    \item [(ii)] The cluster monomials of $\mathcal A$ are linearly independent over $\mathbb Z$.
\end{itemize}
\end{corollary}
\begin{proof}
Thanks to the Laurent phenomenon, we know that each cluster monomial in $\mathcal A$ is a global monomial on $\mathbb A$. Then
the results follow from Corollary \ref{cor:proper} and Theorem \ref{thm:linear-indp}.
\end{proof}

\begin{remark}
 (i) The property in Corollary \ref{cor:A-side} (i) is known as the {\em proper Laurent monomial property} (cf. \cite[Definition 6.1]{cl-2012}), which is a consequence of the global monomial avoidance phenomenon for positive spaces.

(ii) The proper Laurent monomial property and the linear independence of cluster monomials have been proved for skew-symmetric cluster algebras in \cite{cklp-2013} via categorification, for skew-symmetrizable cluster algebras in \cite{ghkk18} via scattering diagrams, and for all {\em acyclic} totally sign-skew-symmetric cluster algebras in \cite{Huang-2019} using the unfolding approach based on the results in \cite{cklp-2013}.

(iii) We note that positive spaces also arise naturally in the theory of generalized cluster algebras introduced by Chekhov and Shapiro \cite{cs-2014}, and our general results also apply to generalized cluster algebras whose clusters are related by subtraction-free birational transformations.
\end{remark}

Finally, we give a description of the global monomials on positive ${\bf A}$-spaces, whose proof is a simple application of the Laurent phenomenon.
\begin{proposition}\label{pro:global-A-monomial}
Let $\mathbb A=(\mathbb F, \{{\bf x}_t\}_{t\in\mathbb T_n})$ be a positive ${\bf A}$-space associated with a cluster pattern $\mathcal S_X=\{({\bf x}_t,  \widetilde B_t)\mid t\in \mathbb T_n\}$,
and  let $\mathcal U^{\rm mono}(\mathbb A)$ be the set of global monomials on $\mathbb A$. Then we have
\[\mathcal U^{\rm mono}(\mathbb A)=\bigcup_{t\in\mathbb T_n}\{ {\bf x}_t^{{\bf h}}\mid {\bf h}=(h_1,\ldots,h_m)^T\in\mathbb Z^m\;\;\text{with}\;\;h_k\geq 0\;\;\text{for}\;\;k\in[1,n]\}.
\]
In particular, global monomials on $\mathbb A$ and cluster monomials in $\mathcal U(\mathbb A)$ are the same up to frozen factors.
\end{proposition}
\begin{proof}
The proof is similar to that of \cite[Lemma 4.2]{cdsgl-2023}, where the setting is positive ${\bf A}$-spaces with no frozen variables.
 For the convenience of the reader, we give a proof here. 
 The containment $\supseteq$ is a consequence of the Laurent phenomenon; we therefore focus on the other direction. Let $u\in \mathcal U^{\rm mono}(\mathbb A)$ be a global monomial on $\mathbb A$.  Then it is of the form $u={\bf x}_t^{\bf h}$ for some $t\in\mathbb T_n$ and ${\bf h}=(h_1,\ldots,h_m)^T\in\mathbb Z^m$. Assume there exists a unfrozen direction $k\in [1,n]$ such that $h_k<0$. We consider the ${\bf A}$-mutation $({\bf x}_{t'}, \widetilde B_{t'})=\mu_k({\bf x}_t, \widetilde B_t)$. We have
 $$x_{i;t} = \begin{cases}
x_{i;t'}, &\text{if } i \neq k, \\
{P}_{k;t'}/x_{k;t'}, &\text{if } i = k,
\end{cases}$$
where $P_{k;t'}\coloneqq \prod_{j=1}^mx_{j;t'}^{[b_{jk}^{t'}]_+}+\prod_{j=1}^mx_{j;t'}^{[-b_{jk}^{t'}]_+}$ is a non-invertible element in $\mathbb Z[x_{1;t'}^{\pm 1},\ldots, x_{m;t'}^{\pm 1}]$.
Then we can see that the expansion of $u={\bf x}_t^{\bf h}$ with respect to the cluster ${\bf x}_{t'}$ can not be a Laurent polynomial in $\mathbb Z[x_{1;t'}^{\pm 1},\ldots, x_{m;t'}^{\pm 1}]$. Hence, we must have $h_k\geq 0$ for $k\in[1,n]$.
\end{proof}

\subsection{Positive spaces from $Y$-patterns}
In this subsection, we give examples of positive spaces from $Y$-patterns \cite{fomin_zelevinsky_2007}. We fix a positive integer $n$ and let $\mathbb F$ be the field of rational functions in $n$ variables over $\mathbb Q$.

A {\em $Y$-seed} in $\mathbb F$ is a pair $({\bf y}, B)$, where \begin{itemize}
	\item ${\bf y} = (y_1, \ldots, y_{n})$ is an ordered set of free generators of $\mathbb F$ over $\mathbb Q$;
	\item  $B=(b_{ij})$ is an $n\times n$ skew-symmetrizable matrix\footnote{Here we follow the setting of \cite{fomin_zelevinsky_2007}, where $Y$-patterns are defined for skew-symmetrizable matrices.}.
\end{itemize}
In this case, the ordered set ${\bf y}$ is called the {\em $Y$-cluster} of $({\bf y}, B)$ and the elements of ${\bf y}$ are called {\em $y$-variables}.

 The {\em ${\bf Y}$-mutation} of a $Y$-seed $({\bf y}, B)$ in direction $k\in[1,n]$ is the pair  $({\bf y}', B') \coloneqq   \mu_k({\bf y},  B)$  given by $ B'=\mu_k(B)$ and
\begin{align}\label{eqn:y-mutation}
 y_i^\prime&=\begin{cases}y_k^{-1}, &\text{if}\;i=k; \\
y_iy_k^{[b_{ki}]_+}(1+y_k)^{- b_{ki}},&\text{otherwise}.\end{cases}
\end{align}
One can check that the new pair  $({\bf y}', B')=\mu_k({\bf y},  B)$ is still a $Y$-seed in $\mathbb F$ and we have $\mu_k^2({\bf y},  B)=({\bf y},  B)$. Since the relation in \eqref{eqn:y-mutation} is subtraction-free, we have $\mathbb Q_{\rm sf}({\bf y})=\mathbb Q_{\rm sf}({\bf y}')$.

\begin{definition}[$Y$-pattern] A {\em $Y$-pattern} $\mathcal S_Y=\{({\bf y}_t,  B_t)\mid t\in \mathbb T_n\}$ in $\mathbb F$
	is an assignment of a $Y$-seed $({\bf y}_t, B_t)$ to
 	every vertex $t$ of $\mathbb T_n$ such that $({\bf y}_{t'},  B_{t'})$ is the ${\bf Y}$-mutation of $({\bf y}_t,  B_t)$ in direction $k$ whenever
	\begin{xy}(0,1)*+{t}="A",(10,1)*+{t'}="B",\ar@{-}^k"A";"B" \end{xy} in $\mathbb T_n$.
\end{definition}
Each $Y$-pattern  $\mathcal S_Y=\{({\bf y}_t,  B_t)\mid t\in \mathbb T_n\}$ in $\mathbb F$ gives rise to a positive space
 \[\mathbb Y\coloneqq (\mathbb F, \{{\bf y}_t\}_{t\in\mathbb T_n}).\]
Such a positive space is called a {\em positive ${\bf Y}$-space}.
Its coordinate algebra $\mathcal U(\mathbb Y)$ is given by
\[
\mathcal U(\mathbb Y)=\bigcap_{t\in\mathbb T_n}\mathbb Z[y_{1;t}^{\pm 1},\ldots,y_{n;t}^{\pm 1}].
\]

\begin{example}
\label{ex:Y-A2}
   Take $B_{t_0}=\begin{bmatrix}
    0&1\\-1&0
\end{bmatrix}$ and ${\bf y}_{t_0}=(y_1,y_2)$. 
It is easy to  check that the $Y$-pattern defined by the initial $Y$-seed $({\bf y}_{t_0}, B_{t_0})$  has only five $Y$-clusters up to permutation:
\[
\begin{array}{rclcrcl}
\mathbf{y}_{t_0} &=& (y_1,\; y_2), & \qquad &
\mathbf{y}_{t_1} &=& \left( y_1^{-1},\; \dfrac{y_1 y_2}{1 + y_1} \right), \\[8pt]
\mathbf{y}_{t_2} &=& \left( \dfrac{y_2}{1 + y_1 + y_1 y_2},\; \dfrac{1 + y_1}{y_1 y_2} \right), & \qquad &
\mathbf{y}_{t_3} &=& \left( \dfrac{1 + y_1 + y_1 y_2}{y_2},\; \dfrac{1}{y_1 (1 + y_2)} \right), \\[8pt]
\mathbf{y}_{t_4} &=& \left( \dfrac{1}{y_2},\; y_1 (1 + y_2) \right). & \qquad &

\end{array}
\]
It can be seen that the $y$-variable $\frac{y_1y_2}{1+y_1}$ in ${\bf y}_{t_1}$ does not belong to $\mathcal U(\mathbb Y)$ and thus it is not a global monomial on $\mathbb Y$.
\end{example}
As the above example shows, the $y$-variables do not, in general, satisfy the Laurent phenomenon.
This is one of the key differences between the cluster patterns (positive ${\bf A}$-spaces) and the $Y$-patterns (positive ${\bf Y}$-spaces).

The following result gives a description of the global monomials on positive ${\bf Y}$-spaces.
\begin{proposition}[{\cite[Lemma 4.3]{cdsgl-2023}}] 
Let $\mathbb Y=(\mathbb F, \{{\bf y}_t\}_{t\in\mathbb T_n})$ be a positive ${\bf Y}$-space associated with a $Y$-pattern $\mathcal S_Y=\{({\bf y}_t,  B_t)\mid t\in \mathbb T_n\}$, and let $\mathcal U^{\rm mono}(\mathbb Y)$ be the set of global monomials on $\mathbb Y$. Then we have
\[
\mathcal U^{\rm mono}(\mathbb Y)=\bigcup_{t\in\mathbb T_n}\{ {\bf y}_t^{{\bf h}}\mid {\bf h}\in\mathbb Z^n\;\;\text{with}\;\;B_t{\bf h}\in\mathbb N^n\}.
\]
\end{proposition}

The following result is a direct consequence of Corollary \ref{cor:proper} and Theorem \ref{thm:linear-indp}.
\begin{corollary}\label{cor:Y-side}
Let $\mathbb Y=(\mathbb F, \{{\bf y}_t\}_{t\in\mathbb T_n})$ be a positive ${\bf Y}$-space. Fix a $Y$-cluster ${\bf y}_{t_0}$, and let $M(t_0)\subseteq \mathcal U^{\rm mono}(\mathbb Y)$ denote the set of global monomials in ${\bf y}_{t_0}$. Then:
\begin{itemize}
    \item [(i)] If $u$ is a global monomial on $\mathbb Y$ and $u\notin M(t_0)$, then no global monomial from  ${\bf y}_{t_0}$ appears in the Laurent expansion of $u$ with respect to ${\bf y}_{t_0}$.
    \item[(ii)] The global monomials on $\mathbb Y$ are linearly independent over $\mathbb Z$.
\end{itemize}

\end{corollary}

\subsection{Positive spaces from Laurent phenomenon algebras}

In this subsection, we give examples of positive spaces from Laurent phenomenon (LP) algebras \cite{lp-2016}. We fix a pair of integers $(n,m)$ with $0< n\leq m$ and let $\mathbb F$ be the field of rational functions in $m$ variables over $\mathbb Q$. 

A {\it{ (LP) seed}} of rank $n$ in $\mathbb F$ is a pair $({\bf x},{\bf P})$, where
\begin{itemize}
    \item ${\bf x}= (x_1, \ldots, x_{m})$ is an ordered set of free generators of $\mathbb F$ over $\mathbb Q$;
	\item $\mathbf{P} = (P_1, \ldots, P_{n})$ is an ordered set of polynomials in $\mathbb Z[x_{n+1}^{\pm 1},\ldots, x_m^{\pm 1}][x_1, \ldots, x_{n}]$ satisfying: 
    
    (LP1) $P_i$ is irreducible in $\mathbb Z[x_{n+1}^{\pm 1},\ldots, x_m^{\pm 1}][x_1, \ldots, x_{n}]$;
    
    (LP2) $P_i$ is not divisible by any $x_j$ for $j\in[1,n]$;
	 
	(LP3) $P_i$ does not depend on $x_i$.
\end{itemize}
In this case, the ordered set ${\bf x}$ is called the {\it cluster} of the seed $({\bf x},{\bf P})$. The elements of ${\bf x}$ are called
{\it cluster variables}. More precisely, we call $x_1,\ldots,x_n$ {\em unfrozen cluster variables} and $x_{n+1},\ldots,x_{m}$ {\em frozen (cluster) variables}\footnote{The frozen variables are used to provide the ground ring $S\coloneqq \mathbb Z[x_{n+1}^{\pm 1},\ldots,x_m^{\pm 1}]$.}. The polynomials $P_1,\ldots,P_n$ are called {\em exchange polynomials}.

Let $P$ and $N$ be two rational functions in $\mathbb Q(x_1,\ldots,x_m)$. We denote by $P|_{x_j\leftarrow N}$ the expression obtained by substituting $x_j$ in $P$ by $N$, and denote by
\[ \mathbb Z[x_1,\ldots,\hat x_j,\ldots, x_n]\coloneqq \mathbb Z[x_1,\ldots,x_{j-1},x_{j+1},\ldots, x_n].
\]

Let $({\bf x},{\bf P})$ be a seed of rank $n$ in $\mathbb F$. The {\em ${\bf LP}$-mutation} of $({\bf x},{\bf P})$ in direction $k\in[1,n]$ is the new seed $({\bf x}', {\bf P}')\coloneqq \mu_k({\bf x},{\bf P})$ defined by the following sequence of steps:

\begin{enumerate}
\item We define an {\em exchange Laurent polynomial} $\widehat P_k=P_k/\prod_{j=1}^nx_j^{a_j}\in\mathbb Z[x_1^{\pm 1},\ldots,x_m^{\pm 1}]$, where 
\[
a_j\coloneqq \begin{cases}
    0,&j=k,\\
    {\rm max}\{s\in\mathbb N \mid \frac{P_k|_{x_j\leftarrow P_j/x_j}}{P_j^s} \in\mathbb Z[x_1^{\pm 1},\ldots,x_m^{\pm 1}]\}, &k\neq j\in [1,n].
\end{cases}
\]

\item The new cluster ${\bf x}' = (x_1', \ldots, x_m')$ is given by:
$$x_i' = \begin{cases}
x_i, &\text{if } i \neq k, \\
\widehat{P}_k/x_k, &\text{if } i = k.
\end{cases}$$

\item For $j\in[1,n]$, define $G_j:=P_j|_{x_k\leftarrow N_j}$, where $N_j=\frac{\widehat P_k|_{x_j\leftarrow 0}}{x_k^{\prime}}\in \mathbb Z[x_1'^{\pm 1},\ldots \hat x_j',\ldots, x_m'^{\pm 1}]$.
We know that $G_j\in \mathbb Z[x_1'^{\pm 1},\ldots  \hat x_j',\ldots, x_m'^{\pm 1}]$.

\item For $j\in[1,n]$, define $H_j=G_j/D_j\in \mathbb Z[x_1'^{\pm 1},\ldots  \hat x_j',\ldots, x_m'^{\pm 1}]$,
where $D_j$ is the greatest common divisor of $G_j$ and $\widehat P_k|_{x_j\leftarrow 0}$ in $\mathbb Z[x_1'^{\pm 1},\ldots  \hat x_j',\ldots, x_m'^{\pm 1}]$.
Note that $D_j$ is defined up to a unit in  $\mathbb Z[x_1'^{\pm 1},\ldots  \hat x_j',\ldots, x_m'^{\pm 1}]$.

\item For $j\in [1,n]$, define the new exchange polynomial $P'_j = M_j H_j$, where $M_j$ is the unique monoic Laurent monomial in $x_1',\ldots,x_n'$ such that 
$$P_j'\in \mathbb Z[x_{n+1}'^{\pm 1},\ldots,x_m'^{\pm 1}][x_1',\ldots,x_n']=\mathbb Z[x_{n+1}^{\pm 1},\ldots,x_m^{\pm 1}][x_1',\ldots,x_n']$$ and is not divisible by any $x_i'$ with $i\in[1,n]$. Note that $P_j'$ is defined up to a unit in $\mathbb Z[x_{n+1}^{\pm 1},\ldots,x_m^{\pm 1}]$.

\item The new seed is given by $ ({\bf x}', {\bf P}') = ((x_1',\ldots,x_m'), (P'_1, \ldots, P'_{n}))$.
\end{enumerate}

All the necessary existence and uniqueness conditions to show that the above mutation gives a unique valid seed can be found in \cite[\S 2]{lp-2016}. 
\begin{remark}\label{rmk:LP}
(i) If $P_j$ does not depend on $x_k$, then $G_j = H_j = P_j$ and $M_j = 1$, implying that $P_j' = P_j$. In this case, the well-definedness of $N_j$ is irrelevant.

(ii) If $\{P_1,\ldots,P_n\}\subseteq \mathbb Q_{\rm sf}({\bf x})$, we see that $ \mathbb Q_{\rm sf}({\bf x})= \mathbb Q_{\rm sf}({\bf x}')$. Moreover, we can choose the greatest common divisor $D_j$ in Step (4) such that the new exchange polynomial $P_j'$ belongs to $\mathbb Q_{\rm sf}({\bf x}')$ for $j\in[1,n]$.
\end{remark}

\begin{definition} [$LP$-pattern and positive $LP$-pattern] \hfill

(i) A {\em $LP$-pattern} $\mathcal S_{LP}=\{({\bf x}_t,  {\bf P}_t)\mid t\in \mathbb T_n\}$ in $\mathbb F$
	is an assignment of a seed $({\bf x}_t,  {\bf P}_t)$ to
 	every vertex $t$ of $\mathbb T_n$ such that $({\bf x}_{t'},   {\bf P}_{t'})$ is obtained from  $({\bf x}_t,  {\bf P}_t)$ by ${\bf LP}$-mutation in direction $k$ whenever
	\begin{xy}(0,1)*+{t}="A",(10,1)*+{t'}="B",\ar@{-}^k"A";"B" \end{xy} in $\mathbb T_n$.

    (ii) A $LP$-pattern  $\mathcal S_{LP}=\{({\bf x}_t,  {\bf P}_t)\mid t\in \mathbb T_n\}$ is said to be {\em positive}, if $\mathbb Q_{\rm sf}({\bf x}_t)=\mathbb Q_{\rm sf}({\bf x}_{t'})$ for any two vertices $t,t'\in\mathbb T_n$.
\end{definition}

We note that the frozen variables are the same in any two clusters of $\mathcal S_{LP}$. Thanks to Remark \ref{rmk:LP} (ii), examples of positive $LP$-patterns can be easily constructed.

\begin{definition}
    Let  $\mathcal S_{LP}=\{({\bf x}_t,  {\bf P}_t)\mid t\in \mathbb T_n\}$  be a $LP$-pattern in $\mathbb F$.
    
\begin{itemize}
    \item [(i)] The {\em Laurent phenomenon (LP) algebra} $\mathcal A$ associated with the $LP$-pattern $\mathcal S_{LP}$ is the $\mathbb Z$-subalgebra of $\mathbb F$ generated by all cluster variables (frozen and unfrozen), together with the inverse of frozen variables, i.e.,
 $$\mathcal A\coloneqq \mathbb Z[x_{1;t},\ldots,x_{n;t},x_{n+1;t}^{\pm 1},\ldots,x_{m;t}^{\pm 1}\mid t\in\mathbb T_n].$$
 \item[(ii)] The {\em upper Laurent phenomenon (LP) algebra} $\mathcal U$ associated with the $LP$-pattern $\mathcal S_{LP}$ is defined by the following intersection of Laurent polynomial algebras:
 \[
 \mathcal U\coloneqq \bigcap_{t\in\mathbb T_n}\mathbb Z[x_{1;t}^{\pm 1},\ldots,x_{m;t}^{\pm 1}].\]
 \item [(iii)] If the $LP$-pattern $\mathcal S_{LP}$ is positive, then it gives rise to a positive space
 \[\mathbb L\coloneqq (\mathbb F, \{{\bf x}_t\}_{t\in\mathbb T_n}).\]
 Such a positive space is called a {\em positive ${\bf LP}$-space}. 
\end{itemize}
\end{definition}
We see that the coordinate algebra $\mathcal U(\mathbb L)$ of a positive ${\bf LP}$-space $\mathbb L$ is exactly the corresponding upper LP algebra.

\begin{theorem}[{\cite[Theorem 5.1]{lp-2016}}] \label{LPLP}
Let $\mathcal{A}$ be a Laurent phenomenon algebra and
$z$ a cluster variable of $\mathcal A$. Then for any seed
$({\bf x}_t,{\bf P}_t)$ of $\mathcal{A}$, we have $z\in \mathbb Z[x_{1;t}^{\pm 1}, \ldots, x_{m;t}^{\pm 1}]$. In particular, $\mathcal A$ is contained in the corresponding upper LP algebra $\mathcal U$.
\end{theorem}

\begin{example} Take ${\bf x}_{t_0}=(x_1,x_2)$ and ${\bf P}_{t_0}=(x_2+2,x_1+3)$. It can be checked that the $LP$-pattern defined by the initial seed $({\bf x}_{t_0},{\bf P}_{t_0})$ has only five seeds (up to permutation):
    \[
\begin{array}{rclcrcl}
\mathbf{x}_{t_0} &=& (x_1,\; x_2), & \qquad &
\mathbf{P}_{t_0} &=& (x_2+2,\;x_1+3), \\[8pt]
\mathbf{x}_{t_1} &=& (x_3, \;x_2), & \qquad &
\mathbf{P}_{t_1} &=& (x_2+2,\;3x_3+2), \\[8pt]
\mathbf{x}_{t_2} &=& (x_3,\;x_4), & \qquad &\mathbf{P}_{t_2} &=& (x_4+1,\;3x_3+2), 
\end{array}
\]
\[\begin{array}{rclcrcl}
\mathbf{x}_{t_3} &=& (x_5,\;x_4), & \qquad &\mathbf{P}_{t_3} &=& (x_4+1,\;2x_5+3),\\[8pt]
\mathbf{x}_{t_4} &=& (x_5,\;x_1), & \qquad &\mathbf{P}_{t_4} &=& (x_1+3,\;2x_5+3),
\end{array}
\]
\[\text{where}\quad\quad x_3:=   \frac{x_2+2}{x_1}, \;\;x_4:=   \frac{2x_1+3x_2+6}{x_1x_2}, \;\;x_5:=  \frac{x_1+3}{x_2}.\]
It can be seen that this $LP$-pattern is a positive $LP$-pattern.
\end{example}

Recall that a {\em cluster monomial} is a monomial in cluster variables from the same cluster.  Our general results on positive spaces yield the following results for LP algebras associated with positive $LP$-patterns, which seem to be new.

\begin{corollary}\label{cor:LP-side}
Let $\mathcal A$ be a LP algebra associated with a 
positive $LP$-pattern, and let $\mathbb L=(\mathbb F, \{{\bf x}_t\}_{t\in\mathbb T_n})$ be the corresponding positive ${\bf LP}$-space. Fix a cluster ${\bf x}_{t_0}$. Then:
\begin{itemize}
    \item[(i)] If $u$ is a cluster monomial and  $u\notin \{{\bf x}_{t_0}^{\bf h}\mid {\bf h}\in\mathbb N^m\}$, then the Laurent expansion of $u$ with respect to ${\bf x}_{t_0}$ is a linear combination of proper Laurent monomials in ${\bf x}_{t_0}$.
    \item [(ii)] The cluster monomials of $\mathcal A$ are linearly independent over $\mathbb Z$.
\end{itemize}
\end{corollary}
\begin{proof}
Thanks to the Laurent phenomenon, we know that each cluster monomial in $\mathcal A$ is a global monomial on $\mathbb L$. Then
the results follow from Corollary \ref{cor:proper} and Theorem \ref{thm:linear-indp}.
\end{proof}

Finally, we give a description of the global monomials on positive ${\bf LP}$-spaces.

\begin{proposition}
Let  $\mathbb L=(\mathbb F, \{{\bf x}_t\}_{t\in\mathbb T_n})$ be a positive ${\bf LP}$-space  associated with a positive $LP$-pattern  $\mathcal S_{LP}=\{({\bf x}_t,  {\bf P}_t)\mid t\in \mathbb T_n\}$, and  let $\mathcal U^{\rm mono}(\mathbb L)$ be the set of global monomials on $\mathbb L$.  Then we have
\[\mathcal U^{\rm mono}(\mathbb L)=\bigcup_{t\in\mathbb T_n}\{ {\bf x}_t^{{\bf h}}\mid {\bf h}=(h_1,\ldots,h_m)^T\in\mathbb Z^m\;\;\text{with}\;\;h_k\geq 0\;\;\text{for}\;\;k\in[1,n]\}.
\]
In particular, global monomials on $\mathbb L$ and cluster monomials in $\mathcal U(\mathbb L)$ are the same up to frozen factors.
\end{proposition}
\begin{proof}
The proof is similar to that of Proposition \ref{pro:global-A-monomial}.
\end{proof}

\section*{Acknowledgements}
The author would like to thank Antoine de Saint Germain and Jiang-Hua Lu for their collaboration for the paper \cite{cdsgl-2023}. The definition of positive spaces used in this work is taken from that paper.

\bibliography{myref}

\begin{bibdiv}
\begin{biblist}

\bib{bfz_2005}{article}{
      author={Berenstein, A.},
      author={Fomin, S.},
      author={Zelevinsky, A.},
       title={{Cluster algebras {III}: Upper bounds and double Bruhat cells}},
        date={2005},
     journal={Duke Math. J.},
      volume={126},
      number={1},
       pages={1\ndash 52},
         url={https://doi.org/10.1215/S0012-7094-04-12611-9},
}

\bib{BK-2012}{article}{
      author={Baumann, P.},
      author={Kamnitzer, J.},
       title={Preprojective algebras and {MV} polytopes},
        date={2012},
        ISSN={1088-4165},
     journal={Represent. Theory},
      volume={16},
       pages={152\ndash 188},
}

\bib{Cao-2023}{unpublished}{
      author={Cao, P.},
       title={{F-invariant in cluster algebras}},
        date={2023},
        note={arXiv:2306.11438v4},
}

\bib{cao-2026b}{unpublished}{
      author={Cao, P.},
       title={{Newton polytopes in cluster algebras and $\tau$-tilting theory}},
        date={2026},
        note={arXiv:2602.07929},
}

\bib{cdsgl-2023}{unpublished}{
      author={Cao, P.},
      author={de~St.~Germain, A.},
      author={Lu, J.H.},
       title={{Tropical friezes and cluster-additive functions via Fock-Goncharov duality and a conjecture of Ringel}},
        date={2023},
        note={arXiv:2311.17712},
}

\bib{CGGLSS-2025}{article}{
      author={Casals, R.},
      author={Gorsky, E.},
      author={Gorsky, M.},
      author={Le, I.},
      author={Shen, L.},
      author={Simental, J.},
       title={Cluster structures on braid varieties},
        date={2025},
     journal={J. Amer. Math. Soc.},
      volume={38},
      number={2},
       pages={369\ndash 479},
}

\bib{ck-2008}{article}{
      author={Caldero, P.},
      author={Keller, B.},
       title={{From triangulated categories to cluster algebras}},
        date={2008},
     journal={Invent. Math.},
      volume={172},
       pages={169\ndash 211},
}

\bib{cs-2014}{article}{
      author={Chekhov, L.},
      author={Shapiro, M.},
       title={Teichm\"uller spaces of {R}iemann surfaces with orbifold points of arbitrary order and cluster variables},
        date={2014},
     journal={Int. Math. Res. Not. IMRN},
      number={10},
       pages={2746\ndash 2772},
}

\bib{dtw-2025}{article}{
      author={e~Moura, G. Z.~Dantas},
      author={Kandalam, R. C.~Telekicherla},
      author={Woodruff, D.},
       title={Cluster monomials in graph {L}aurent phenomenon algebras},
        date={2025},
     journal={Algebr. Comb.},
      volume={8},
      number={4},
       pages={997\ndash 1019},
}

\bib{fg-2009}{article}{
      author={Fock, V.},
      author={Goncharov, A.},
       title={{Cluster ensembles, quantization and the dilogarithm}},
        date={2009},
     journal={{Ann. Sci. \'{E}c. Norm. Sup\'{e}r.}},
      volume={42},
      number={6},
       pages={865\ndash 930},
}

\bib{fz_2002}{article}{
      author={Fomin, S.},
      author={Zelevinsky, A.},
       title={{Cluster algebras {I}: Foundations}},
        date={2002},
     journal={J. Amer. Math. Soc.},
      volume={15},
       pages={497\ndash 529},
}

\bib{fz-2003-notes}{incollection}{
      author={Fomin, S.},
      author={Zelevinsky, A.},
       title={Cluster algebras: notes for the {CDM}-03 conference},
        date={2003},
   booktitle={Current developments in mathematics, 2003},
   publisher={Int. Press, Somerville, MA},
       pages={1\ndash 34},
}

\bib{fomin_zelevinsky_2007}{article}{
      author={Fomin, S.},
      author={Zelevinsky, A.},
       title={Cluster algebras {IV}: Coefficients},
        date={2007},
     journal={Compos. Math.},
      volume={143},
      number={1},
       pages={112–164},
}

\bib{ghkk18}{article}{
      author={Gross, M.},
      author={Hacking, P.},
      author={Keel, S.},
      author={Kontsevich, M.},
       title={Canonical bases for cluster algebras},
        date={2018},
     journal={J. Amer. Math. Soc.},
      volume={31},
      number={2},
       pages={497\ndash 608},
}

\bib{GKZ-1994}{book}{
      author={Gelfand, I.~M.},
      author={Kapranov, M.~M.},
      author={Zelevinsky, A.~V.},
       title={Discriminants, resultants, and multidimensional determinants},
      series={Mathematics: Theory \& Applications},
   publisher={Birkh\"auser Boston, Inc., Boston, MA},
        date={1994},
        ISBN={0-8176-3660-9},
}

\bib{gls_2011}{article}{
      author={Gei{\ss}, C.},
      author={Leclerc, B.},
      author={Schr\"{o}er, J.},
       title={Kac-{M}oody groups and cluster algebras},
        date={2011},
        ISSN={0001-8708},
     journal={Adv. Math.},
      volume={228},
      number={1},
       pages={329\ndash 433},
         url={https://www.sciencedirect.com/science/article/pii/S0001870811001617},
}

\bib{GLSB-2026}{article}{
      author={Galashin, P.},
      author={Lam, T.},
      author={Sherman-Bennett, M.},
       title={Braid variety cluster structures, {II}: general type},
        date={2026},
     journal={Invent. Math.},
      volume={243},
      number={3},
       pages={1079\ndash 1127},
}

\bib{Huang-2019}{article}{
      author={Huang, M.},
       title={Proper {L}aurent monomial property of acyclic cluster algebras},
        date={2019},
     journal={Comm. Algebra},
      volume={47},
      number={9},
       pages={3520\ndash 3526},
}

\bib{cklp-2013}{article}{
      author={Irelli, G.~Cerulli},
      author={Keller, B.},
      author={Labardini-Fragoso, D.},
      author={Plamondon, P.-G.},
       title={Linear independence of cluster monomials for skew-symmetric cluster algebras},
        date={2013},
     journal={Compos. Math.},
      volume={149},
      number={10},
       pages={1753\ndash 1764},
}

\bib{cl-2012}{article}{
      author={Irelli, G.~Cerulli},
      author={Labardini-Fragoso, D.},
       title={Quivers with potentials associated to triangulated surfaces, {P}art {III}: tagged triangulations and cluster monomials},
        date={2012},
     journal={Compos. Math.},
      volume={148},
      number={6},
       pages={1833\ndash 1866},
}

\bib{Scott-2006}{article}{
      author={J.~S., Scott},
       title={Grassmannians and cluster algebras},
        date={2006},
     journal={Proc. London Math. Soc. (3)},
      volume={92},
      number={2},
       pages={345\ndash 380},
}

\bib{lp-2016}{article}{
      author={Lam, T.},
      author={Pylyavskyy, P.},
       title={Laurent phenomenon algebras},
        date={2016},
     journal={Camb. J. Math.},
      volume={4},
      number={1},
       pages={121\ndash 162},
}

\bib{Schneider_2013}{book}{
      author={Schneider, R.},
       title={{Convex Bodies: The Brunn–Minkowski Theory}},
     edition={2},
      series={Encyclopedia of Mathematics and its Applications},
   publisher={Cambridge University Press},
        date={2013},
}

\end{biblist}
\end{bibdiv}
\end{document}